\documentclass{smfart}
\usepackage[latin1]{inputenc}
\usepackage[T1]{fontenc}
\usepackage[french]{babel}
\usepackage{amsfonts}
\usepackage{amsmath}
\usepackage{amsthm}
\usepackage{amssymb}
\usepackage{amsopn}
\usepackage{upref}
\usepackage{mathrsfs}
\usepackage{ae, aecompl}
\usepackage{graphicx}
\usepackage{fancyhdr}
\usepackage{lmodern}
\usepackage{babelbib}
\usepackage{enumerate}
\usepackage{url}
\selectbiblanguage{french}
\usepackage[all,cmtip]{xy}
\usepackage{dsfont}

\DeclareMathOperator{\Rep}{Rep}

\DeclareMathOperator{\GL}{GL}
\DeclareMathOperator{\cl}{cl}
\DeclareMathOperator{\ab}{ab}
\DeclareMathOperator{\CH}{CH}
\DeclareMathOperator{\CHM}{CHM}
\DeclareMathOperator{\DM}{DM}
\DeclareMathOperator{\NumM}{NumM}
\DeclareMathOperator{\VHS}{VHS}

\DeclareMathOperator{\PL}{PL}
\DeclareMathOperator{\Ort}{O}

\DeclareMathOperator{\Immaginario}{Im}

\DeclareMathOperator{\interiore}{int}

\DeclareMathOperator{\End}{End}
\DeclareMathOperator{\Hom}{Hom}

\DeclareMathOperator{\Id}{Id}

\DeclareMathOperator{\Lef}{Lef}
\DeclareMathOperator{\op}{op}

\DeclareMathOperator{\Sym}{Sym}

\DeclareMathOperator{\Var}{Var}
\DeclareMathOperator{\Ab}{Ab}

\DeclareMathOperator{\Ve}{Vec}

\renewcommand{\Im}{\Immaginario}

\renewcommand{\int}{\interiore}
\renewcommand{\tilde}{\widetilde}

\renewcommand{\O}{\Ort}

\newenvironment{npar}[1]{\begin{paragrafo_numerato_nome}\textbf{#1.}}{\end{paragrafo_numerato_nome}}

\newtheoremstyle{theorem}{11pt}{11pt}{\itshape}{}{\bfseries}{.}{.5em}{}
\newtheoremstyle{note}{11pt}{11pt}{}{}{\bfseries}{.}{.5em}{}

\theoremstyle{note}
\newtheorem{defin}{Définition}[section]

\newtheorem{rem}[defin]{Remarque}
\newtheorem{paragrafo_numerato}[defin]{}
\newtheorem{paragrafo_numerato_nome}[defin]{}

\theoremstyle{theorem}
\newtheorem{thm}[defin]{Théorème}
\newtheorem{cor}[defin]{Corollaire}
\newtheorem{lem}[defin]{Lemme}
\newtheorem{prop}[defin]{Proposition}

\newcommand{\Q}{\mathbb{Q}}
\newcommand{\Z}{\mathbb{Z}}

\newcommand{\C}{\mathbb{C}}
\newcommand{\isocan}{\xrightarrow{\hspace{1.85pt}\sim \hspace{1.85pt}}}
\numberwithin{equation}{section}
\author{Giuseppe Ancona}

\address{Duisburg-Essen University\\
Fakultät Mathematik 
Universität Duisburg-Essen,  Campus Essen \\
9, Thea-Leymann-Straße  \\
45127 Essen (Allemagne)}
\email{monsieur.beppe@gmail.com}
\urladdr{https://www.uni-due.de/\textasciitilde hm0168/}
  
\begin{document}

\title{Décomposition de motifs abéliens}

\date{}

\begin{abstract}
Soit $A$ une variété abélienne et fixons une cohomologie de Weil à coefficients dans un corps $F$. Soient $H^1(A,F)$ le premier groupe de cohomologie de $A$ et  $\Lef(A) \subset \GL(H^1(A,F))$ son groupe de Lefschetz; c'est-à-dire le sous-groupe de $\GL(H^1(A,F))$ des applications linéaires qui commutent aux endomorphismes de $A$ et qui respectent l'accouplement induit par une polarisation. Nous construisons explicitement une $\Q$-algèbre de correspondances $B_{i,r}$ telle que  l'application classe de cycle  induise un isomorphisme \[\cl_{|_{B_{i,r}}}: B_{i,r} \otimes_{\Q} F \isocan \End_{\Lef(A)}(H^i(A^r,F)).\] 
Nous donnons aussi des versions relatives de ce résultat. Nous en déduisons en particulier le fait suivant. Soit $S=S_K(G,\mathcal{X})$ une variété de Shimura de type PEL. Alors le foncteur \textit{construction canonique} ${\mu: \Rep (G) \rightarrow \VHS(S(\C))}$ se relève en un foncteur  ${\tilde{\mu}: \Rep (G) \rightarrow \CHM(S)_{\Q}}$, où $\CHM(S)_{\Q}$ est la catégorie des  motifs de Chow relatifs.  
 \end{abstract}
 
 \begin{altabstract}
Let $A$ be an abelian variety and let us fix a   Weil cohomology with coefficients in $F$. Let $H^1(A,F)$ be the first cohomology group of $A$ and $\Lef(A) \subset \GL(H^1(A,F))$ be its Lefschetz group, i.e. the sub-group of $\GL(H^1(A,F))$ of linear applications  commuting with endomorphisms of $A$ and respecting the pairing induced by a polarization. We give an explicit presentation of a $\Q$-algebra of correspondences $B_{i,r}$ such that  the cycle  class  map induces an isomorphism \[\cl_{|_{B_{i,r}}}: B_{i,r} \otimes_{\Q} F \isocan \End_{\Lef(A)}(H^i(A^r,F)).\]
We also give relative versions of this result. We deduce in particular the following fact. Let $S=S_K(G,\mathcal{X})$ be a Shimura variety of PEL type.  Then  the functor \textit{canonical construction} ${\mu: \Rep (G) \rightarrow \VHS(S(\C))}$ lifts  to a functor ${\tilde{\mu}: \Rep (G) \rightarrow \CHM(S)_{\Q}}$, where $\CHM(S)_{\Q}$  is the category of relative Chow motives.
 \end{altabstract}

\maketitle

\tableofcontents

\section{Introduction}

Soit $A$ une variété abélienne sur un corps $k$. Considérons une cohomologie de Weil à coefficients dans un corps $F$, par exemple la  cohomologie de Betti ($F=\Q$, $k=\C$) ou la cohomologie $\ell$-adique ($F=\Q_{\ell}$). Soient $H^1(A,F)$ le premier groupe de cohomologie de $A$ et  $\Lef(A) \subset \GL(H^1(A,F))$ son groupe de Lefschetz; c'est-à-dire le sous-groupe de $\GL(H^1(A,F))$ des applications linéaires qui commutent aux endomorphismes de $A$ et qui respectent l'accouplement induit par une polarisation (voir la Définition \ref{gpe lef}). 

Dans  \cite{Milne},  Milne montre que, pour tous $i$ et $r$, tous les éléments de l'algèbre \[\End_{\Lef(A)}(H^i(A^r,F))=\End_{\Lef(A)}(\wedge^i (H^1(A,F)^{\oplus r}))\] sont algébriques, c'est-à-dire qu'ils proviennent de correspondances algébriques via l'application classe de cycle. Mieux: il s'agit exactement des images des combinaisons linéaires d'intersections de diviseurs.\footnote{Ceci généralisait les résultats de Tankeev \cite{Tan}, Ribet \cite{Ribet}, Murty \cite{Mur}, Hazama \cite{Haz}, Ichikawa \cite{Ichi} et Zarhin \cite{Zar}.} Ainsi, l'ensemble des correspondances qui interviennent est décrit comme espace vectoriel.

\

Dans ce travail nous construisons explicitement une $\Q$-algèbre de correspondances $B_{i,r}$ telle que  l'application classe de cycle  induise un isomorphisme \[\cl_{|_{B_{i,r}}}: B_{i,r} \otimes_{\Q} F \isocan \End_{\Lef(A)}(H^i(A^r,F)).\] La surjectivité de cet isomorphisme redonne le résultat d'algébricité de Milne. L'injectivité permet d'avoir des informations non seulement sur le motif homologique de $A^r$ mais aussi sur son motif de Chow. 
\begin{thm}\label{thm intro}
 Toute décomposition cohomologique de $H^i(A^r,F)$ en somme de \linebreak sous-$\Lef(A)$-représentations se relève canoniquement en une décomposition dans la catégorie $\CHM(k)_F$ des motifs de Chow sur $k$ à coefficients dans $F$.
 \end{thm}
  En effet, une telle décomposition de $H^i(A^r,F)$ correspond à une liste $p_1,\ldots,p_j$ de projecteurs orthogonaux de l'algèbre $\End_{\Lef(A)}(H^i(A^r,F))$ et il suffit alors de prendre leurs préimages $\pi_1,\ldots,\pi_j$ dans $B_{i,r}\otimes_{\Q} F$. Le mot \og canoniquement\fg{} dans l'énoncé du théorème ci-dessus (et dans ceux qui suivent) est à entendre au sens suivant: des telles correspondances $\pi_1,\ldots,\pi_j$ non seulement existent, mais elles sont uniques à l'intérieur d'une algèbre explicite.
  
  Le Théorème \ref{thm intro} est un raffinement du résultat de Künnemann \cite{Ku2} qui démontre que les décompositions de Lefschetz se relèvent aux motifs de Chow. 
  
  À titre d'exemple, voici un cas particulier du Théorème \ref{thm intro}.
  
 \begin{cor} Soit $A$ une variété abélienne complexe de dimension première et telle que la $\Q$-algèbre des endomorphismes $\End(A)\otimes_{\Z} \Q$ soit un corps quadratique imaginaire. Alors, pour tous $i$ et $r$, toute décomposition de $H^i(A^r,\Q)$ en somme de sous-$\Q$-structures de Hodge se relève canoniquement en une décomposition dans la catégorie $\CHM(\C)_{\Q}$, des motifs de Chow sur $\C$ à coefficients dans $\Q$.
 \end{cor}
En effet, sous ces hypothèses, le groupe de Mumford-Tate de $A$ coïncide avec le groupe de Lefschetz (fois les homothéties),  \cite{Tan, Ribet}.

\
 
  Les résultats qui précèdent ont des analogues relatifs. Un exemple en est le résultat suivant, qui était la motivation originale de ce travail  (voir aussi le Théorème \ref{thm rep}).
 \begin{thm}\label{automorphe}
 Soit $S=S_K(G,\mathcal{X})$ une variété de Shimura de type PEL. Soient $S_0$ son modèle canonique sur le corps réflexe et $A_0$ le schéma abélien universel sur celui-ci. Soit $f: A_0^r \rightarrow S_0$ le $r$-ième produit fibré  de $A_0$ au-dessus de $S_0$. Alors toute décomposition de  $R^i f_*\Q_{A_0^r}$ induite par le foncteur \og construction canonique\fg{} ${\mu: \Rep (G) \longrightarrow \VHS(S_0(\C))}$ (\cite[1.18]{Pi}, voir aussi \cite[\S 2]{BW}) se relève canoniquement en une décomposition dans la catégorie $\CHM(S_0)_{\Q}$ des motifs de Chow relatifs. 
\end{thm}
Cet énoncé contribue au programme de Wildeshaus \cite{Wild3} qui vise à attacher un motif à une forme modulaire, en généralisant le résultat de Scholl \cite{SchMod} pour les formes modulaires classiques (voir aussi la Section \ref{motif form}). 

\begin{rem}
Dans un travail récent \cite{Moo} Moonen a obtenu des décompositions motiviques assez proches de celles-ci par des méthodes différentes des nôtres.
\end{rem}

\begin{npar}{Stratégies de preuve}
Pour démontrer que l'application $\cl_{|_{B_{i,r}}}$ ci-dessus est un isomorphisme on remarque d'abord que l'énoncé est stable par extension des scalaires. Soit $\overline{F}$ une clôture algébrique de $F$. Le groupe $\Lef(A)\times_F \overline{F}$ devient un produit de groupes linéaires, orthogonaux et symplectiques \cite[Summary p.655]{Milne}. On dispose alors d'une liste explicite de générateurs et relations de l'algèbre $\End_{\Lef(A)}(H^i(A^r,F))$, grâce à la théorie classique des invariants \cite{Weyl}. La surjectivité (respectivement l'injectivité) est obtenue en relevant chaque générateur (respectivement chaque relation). 

L'algèbre $B_{i,r}$ sera, par définition, l'algèbre engendrée par une certaine famille de correspondances. Cette famille sera choisie pour relever les générateurs, ainsi la partie plus délicate est vérifier l'injectivité. Pour une preuve de l'injectivité qui suit le schéma précédent, voir \cite{Anc}. Nous donnons ici une preuve qui demande moins de vérifications (bien qu'elle soit moins élémentaire, nous commentons ce choix dans la Remarque \ref{alternative}). Elle s'appuie sur les résultats récents de O'Sullivan. Grâce à \cite{O'S} on dispose d'un critère général: si une algèbre de correspondances est engendrée par des cycles d'un certain type\footnote{Des tels cycles sont appelés par O'Sullivan \og symmetrically distinguished\fg.} alors l'application classe de cycle, restreinte à cette algèbre, est injective. Il ne nous restera donc qu'à vérifier que les cycles  qui engendrent  $B_{i,r}$ sont bien de ce type. 
\end{npar}

\begin{npar}{Organisation du texte}
Les sections 2, 3 et 4 sont essentiellement de rappel. La section 2 donne les résultats classiques sur les motifs des schémas abéliens, d'après Deninger, Murre, Künnemann et Kings. La section 3 est autour des cohomologies de Weil relatives.  Nous vérifions au passage que la décomposition construite dans \cite{DM} est effectivement une décomposition de Chow-Künneth pour toutes les cohomologies de Weil (dans \cite{DM} ceci avait été vérifié pour la cohomologie $\ell$-adique). Dans la section 4 nous définissons le groupe $\Lef(A)$ et en donnons des propriétés de base.

Les sections 5, 6, 7 et 8 démontrent que l'application $\cl_{|_{B_{i,r}}}$ ci-dessus est un isomorphisme. La section 5 donne la définition d'une algèbre de correspondances convenable $\mathcal{B}_{n,F}$.  La section 6 énonce le théorème principal et démontre la surjectivité de l'application. L'injectivité est démontrée dans la section 7, après un rappel sur le théorème de scindage de O'Sullivan \cite{O'S}. L'algèbre $B_{i,r}\otimes_{\Q} F$ sera construite à partir de $\mathcal{B}_{n,F}$ au \ref{Bir} et ses propriétés principales se déduiront formellement de celles de $\mathcal{B}_{n,F}$. La section 8 termine en étudiant le cas des variété de Shimura et en démontrant  le Théorème \ref{automorphe}.

\end{npar}

\begin{npar}{Remerciements}
 Ce travail a une forte intersection avec ma thèse de doctorat. Je tiens à remercier mon directeur Jörg Wildeshaus pour ses conseils  précieux. Je dois beaucoup aux rapporteurs de thèse Ben Moonen et Rutger Noot pour les échanges enrichissants que l'on a eu et pour tous leurs commentaires. Je remercie les autres membres de jury Christian Ausoni, Spencer Bloch, Jean-Benoît Bost, Michael Harris et Benoît Stroh pour leurs questions et remarques stimulantes. Merci à Olivier Benoist, Yohan Brunebarbe, Ishai Dan-Cohen, Javier Fres\'an, Lie Fu, Florian Ivorra et Marco Maculan pour la relecture des premières versions du texte. J'exprime ma gratitude au referee, ses remarques ont permis d'améliorer le texte. 
\end{npar}

 \begin{npar}{Notations}
 Dans ce qui suit $k$ est un corps de caractéristique quelconque qui n'est pas forcément algébriquement clos. On fixe un schéma de base $S$  qui est connexe quasi-projectif et lisse sur $k$. On notera $\Var/S$ la catégorie des schémas connexes, projectifs et lisses sur $S$; ses objets seront appelés variétés. On notera $\Ab/S$ la catégorie des schémas abéliens, sous-catégorie non-pleine de  $\Var/S$.  Les objets seront les $S$-schémas abéliens et les morphismes seront les morphismes comme $S$-schémas en groupe.
 
 On désignera par $A$ un schéma abélien sur $S$, de dimension relative $g$. On notera $\End(A)=\End_{\Ab/S}(A)$ son anneau des endomorphismes et $\End(A)^{\op}$ l'anneau opposé. Pour tout entier $n$ l'endomorphisme de $A$ de multiplication par $n$ sera noté $n_A.$ Pour tout entier $r\geq 1$, on notera $A^r$ le $r$-ième produit fibré de $A$ au-dessus de $S$. 
 
 Le corps $F$ désignera toujours un corps de caractéristique nulle  qui sera le corps des coefficients (des groupes de Chow ou de la cohomologie de Weil qu'on considérera). Pour toute variété  $X\in \Var/S$, $\CH^i(X)_F:=\CH^i(X)\otimes_{\Z}F$ sera le $F$-module des cycles algébriques de $X$ de codimension $i$ modulo équivalence rationnelle.
 
 Le dual d'un objet $Y$ sera noté $Y^{\vee}$. L'endomorphisme identité d'un objet $Y$ sera noté $\Id_Y$ ou simplement $\Id$. Quand $B$ est un groupe ou un anneau agissant sur un $R$-module $V$, $\End_B(V)$ désignera l'algèbre des endomorphismes $R$-linéaires de $V$ qui commutent à l'action de $B$.
 \end{npar}

\section{Motif d'un schéma abélien}\label{Motif d'un schéma abélien}
Dans cette section nous rappelons les résultats classiques sur le motif d'un schéma abélien.
\begin{npar}{Motifs de Chow relatifs}\label{Motifs de Chow relatifs}
Soit $\CHM(S)_F$ la catégorie  $F$-linéaire, tensorielle, symétrique, rigide et pseudo-abélienne des motifs de Chow relatifs (à coefficients dans $F$), munie du foncteur contravariant, dit \textit{motif}  
\[M : \Var/S \rightarrow \CHM(S)_F,\]
telle que, si $X\in \Var /S $ est une variété de dimension relative $d$ et $Y\in \Var /S$ est une autre variété, on ait
\[\Hom_{\CHM(S)_F}(M(X),M(Y)):=\CH^{d}(X\times_S Y)_F.\]
Pour une introduction à cette catégorie, voir  \cite{DM}, \cite{Ku1} et \cite{CoHa}.

Le motif $M(S)$ est appelé \textit{motif unité} et noté $\mathds{1}.$ Le motif $M(\mathbb{P}_S^1)$ s'écrit canoniquement comme somme directe de $\mathds{1}$ et d'un autre motif, qui est  appelé \textit{motif de Lefschetz} et qui sera noté  $\mathbb{L}$  ou $\mathds{1}(-1)$. Pour tout entier positif $m$ tensoriser par $\mathbb{L}^{\otimes m}$  ou par son dual $(\mathbb{L}^{\vee})^{\otimes m}$ est une auto-équivalence de la catégorie $\CHM(S)_F$ appelée \textit{twist de Tate} et
 notée respectivement $(-m)$ et $(m)$. 
\end{npar}

\begin{npar}{Puissances symétriques}\label{puissances} (Voir aussi \cite[3.3.1]{Andmot}). La structure tensorielle de $\CHM(S)_F$ est construite de sorte qu'on ait tautologiquement la formule de Künneth
\[M(X \times_S Y)=M(X)\otimes M(Y).\]
La structure symétrique du produit tensoriel est par définition donnée par
\[M(\sigma_{XY}): M(X)\otimes M(Y) \isocan M(Y)\otimes M(X),\]
où 
\[\sigma_{XY}: Y \times_S X \isocan X \times_S Y\]
est l'isomorphisme d'échange. Cette condition détermine la structure symétrique pour toute la catégorie $\CHM(S)_F$ puisque tout motif est facteur direct, après twist de Tate, d'un motif de la forme $M(X)$. 

Ainsi, pour tout motif $Z$, on dispose de l'action du groupe des permutations $\mathcal{S}_n$ sur $Z^{\otimes n}$. Par définition $\Sym^n Z$ est le facteur direct de $Z^{\otimes n}$, image du projecteur
\[P_{\Sym^n Z}= \sum_{\sigma \in \mathcal{S}_n} \sigma.\] 
\end{npar}

\begin{thm}[{\cite[Thm. 3.1]{DM}}]\label{DM}
Dans $\CHM(S)_F$ il y a une décomposition 
\[M(A)=\bigoplus_{i=0}^{i=2g}\mathfrak{h}^i(A).\]
 naturelle en $A \in \Ab/S$, et caractérisée par l'égalité (valable pour tout entier $n$) 
 \[M(n_A)= \bigoplus_{i=0}^{i=2g} n^i \Id_{\mathfrak{h}^i(A)}.\]
 
\end{thm}

\begin{thm}[{\cite[Thm. 3.1.1 et 3.3.1]{Ku1}, \cite[p. 85]{Ku2}}]\label{Ku}
\begin{enumerate}[(i)]

\item Pour tout ${0\leq i \leq 2g},$ il existe un isomorphisme canonique

\[\Sym ^i \mathfrak{h}^1(A) \isocan \mathfrak{h}^i(A).\]

\item \label{dimKu} Pour $i> 2g,$ le motif $ \Sym ^i \mathfrak{h}^1(A)$ est nul. Pour $i=0$ on a $\mathfrak{h}^0(A)=\mathds{1},$  et pour $i=2g$ on a $\mathfrak{h}^{2g}(A)=\mathbb{L}^{\otimes g}.$

\item L'isomorphisme  $\Sym ^* \mathfrak{h}^1(A) \isocan M(A)$ respecte les structures multiplicatives.\footnote{Rappelons que pour $M(A)$ il s'agit du morphisme induit par l'inclusion de la diagonale \[ M(A)\otimes M(A)=M(A^2) \longrightarrow M(A).\]}

\item On a la dualité de Poincaré \[\mathfrak{h}^{2g-i}(A)^{\vee}=\mathfrak{h}^i(A)(g)\] et des isomorphismes  \og de Lefschetz \fg{} non canoniques
\[\mathfrak{h}^i(A)\isocan \mathfrak{h}^{2g-i}(A)(g-i).\]

\end{enumerate}
\end{thm}

\begin{rem}
Par la règle des signes de Koszul, le morphisme $\sigma_{XY}$ de \ref{puissances} agit sur la cohomologie avec des signes (qui dépendent du degré). Ceci fait que l'isomorphisme du Théorème \ref{Ku}(i) se réalise dans l'isomorphisme classique d'espaces vectoriels $\wedge^i H^1(A) \isocan H^i(A)$ (comme on verra dans la Proposition \ref{coho abel}(i)). 

Dans l'article original, Künnemann note $ \Sym ^i \mathfrak{h}^1(A)$ plutôt par $ \wedge ^i \mathfrak{h}^1(A)$ \cite[2.6]{Ku1}. Nous préférons les puissances symétriques (et voir le foncteur de réalisation comme un foncteur vers les super-espaces vectoriel, voir la Définition \ref{def coho Weil}). Ceci suit la convention plus moderne utilisée dans \cite{Andmot, Kim, O'S} ainsi que dans la théorie des motifs mixtes.
\end{rem}

\begin{cor}\label{coro Ku}
Il y a des isomorphismes  \og de Poincaré-Lefschetz \fg{} non canoniques
\[\PL:\mathfrak{h}^1(A)\isocan \mathfrak{h}^{1}(A)^{\vee}(-1).\]
\end{cor}

\begin{prop}[{\cite[Prop. 2.2.1]{Ki}}] \label{Ki}
Le foncteur $\mathfrak{h}^1$ induit un isomorphisme de $F$-algèbres
\[\End(A)^{\op}\otimes_{\Z}F \isocan \End_{\CHM(S)_F}(\mathfrak{h}^1(A)).\]
\end{prop}

\section{Cohomologies de Weil relatives}\label{Cohomologies}
Rappelons que $F$ désigne un corps de caractéristique nulle. Soit  $ \Ve^{\pm}_F$ la catégorie $F$-linéaire et tensorielle des $F$-super-espaces vectoriels $\Z$-gradués de dimension finie.  Rappelons que la symétrie dans cette catégorie est donnée par la règle des signes de Koszul.
\begin{defin}\label{def coho Weil}
 Soit  $H^*=\oplus_{i\in \mathbb{N}} H^i$ un foncteur contravariant
\[H^*: \Var /S\longrightarrow  \Ve^{\pm}_F.\]
On pose $F(-1):=H^2(\mathbb{P}_S^1)\in \Ve^{\pm}_F$ et soit $F(1)$ son dual. Pour tout entier positif $m$ l'opération de tensorisation par $F(-1)^{\otimes m}$ ou par $F(1)^{\otimes m}$ est notée respectivement $(-m)$ et $(m)$ et appelée \textit{twist de Tate}.

 Une cohomologie de Weil (relative au-dessus de $S$ à coefficients dans $F$) est un tel foncteur $H^*$ qui vérifie la dualité de Poincaré et la formule de Künneth et qui est muni d'une application classe de cycle
\[\CH^i(X)_F\longrightarrow H^{2i}(X)(i).\]
\end{defin}
\begin{rem} Dans \cite{CDWeil}, Cisinski et Déglise définissent une bonne notion de cohomologie de Weil mixte (relative au-dessus de $S$). Dans ce cas la catégorie $\Var /S$ est élargie aux $S$-schémas lisses mais pas forcément projectifs.
\end{rem}
\begin{npar}{Exemples}\label{exe weil} (i) Supposons $k$ plongé dans $\C$, et fixons un point $s \in S(\C)$. Alors le foncteur
\[(f:X\rightarrow S) \mapsto \bigoplus_i (R^if_* \Q_X)_s\] 
est une cohomologie de Weil pour $F=\Q$. Dans ce cas chaque $\Q$-espace vectoriel $(R^if_* \Q_X)_s$ est muni d'une action du groupe de Mumford-Tate générique de la variation de $\Q$-structures de Hodge $R^if_* \Q_X$ au-dessus de $S(\C)$.

(ii) Fixons un nombre premier $\ell$ différent de la caractéristique de $k$, et fixons  un point géométrique $\bar{s}$ de $S$. Alors le foncteur
\[(f:X\rightarrow S) \mapsto \bigoplus_i (R^if_* (\Q_{\ell})_X)_{\bar{s}}\] 
est une cohomologie de Weil pour $F=\Q_{\ell}$. Dans ce cas chaque $\Q_{\ell}$-espace vectoriel $(R^if_* (\Q_{\ell})_X)_{\bar{s}}$ est muni d'une action du groupe fondamental étale  $\pi_{1}^{\textrm{ét}}(S,\bar{s}).$

\end{npar}
\begin{rem}\label{def real} (i) Si $H^*$ est une cohomologie de Weil à coefficients dans $F$ et $F'$ est une extension du corps $F$ alors $X \mapsto H^*(X)\otimes_F F'$
est une cohomologie de Weil à coefficients dans $F'$.

(ii) Considérons l'exemple \ref{exe weil}(i) (le (ii) est analogue) et soit $H_{B}^{*}(X(\C),\Q)$ la cohomologie  de Betti \og standard\fg{} de $X$. L'application classe de cycle la cohomologie de Weil relative au-dessus de $S$ de l'exemple \ref{exe weil}(i) est ainsi obtenue: on compose l'application classe de cycle \og absolue\fg{} $\CH^i(X)\rightarrow H_{B}^{2i}(X(\C),\Q)(i)$ avec   l'application induite par la suite spectrale de Leray   $H_{B}^{2i}(X(\C),\Q)(i)\rightarrow H^0(R^{2i}f_* \Q_X (i)),$  puis on évalue en $s$ les sections globales de $R^{2i}f_* \Q_X (i)$.

(iii) Soient $H^*$ une cohomologie de Weil à coefficients dans $F$ et $\CHM(S)_F$ la catégorie des motifs de Chow relatifs à coefficients dans $F$ (Définition \ref{Motifs de Chow relatifs}). Soient $X\in \Var /S $ une variété de dimension relative $d$ et $Y\in \Var /S$ une autre variété. Définissons l'application 
\[ \Hom(M(X),M(Y))=\CH^{d}(X \times_S Y )_F  \longrightarrow H^{2d}(X \times_S Y)(d)= \Hom(H^*(X),H^*(Y)),\] 
où la première égalité vient de la construction des motifs de Chow (\ref{Motifs de Chow relatifs}), la dernière vient de la formule de Künneth et de la dualité de Poincaré et l'application au centre est l'application classe de cycle.

Il suit formellement des Définitions \ref{Motifs de Chow relatifs} et \ref{def coho Weil} que ceci s'étend en un unique foncteur dit \og de réalisation\fg
\[R: \CHM(S)_F \longrightarrow \Ve^{\pm}_F\]
covariant, $F$-linéaire, tensoriel symétrique, qui commute aux twists de Tate et qui vérifie
\[R \circ M = H^*, \]
où $M : \Var/S \rightarrow \CHM(S)_F$ est le foncteur motif (\ref{Motifs de Chow relatifs}).

\end{rem} 
\begin{prop}\label{coho abel}
Soient $H^*$ une cohomologie de Weil, $R$ le foncteur de réalisation induit (\ref{def real}(iii)) et $\mathfrak{h}^i(A)$ les motifs définis dans le Théorème \ref{DM}. Alors:

(i) pour tout $i$ on a l' égalité $R(\mathfrak{h}^i(A))=H^i(A),$  

(ii) le $F$-espace vectoriel $R(\mathfrak{h}^1(A))=H^1(A)$ est de dimension $2g,$ 

(iii) pour tout $f \in \End(A)$, le polynôme caractéristique de $H^1(f)$ ne dépend pas de la cohomologie de Weil $H^*$,

(iv) l'application $\End(A)^{\op}\otimes_{\Z} F \longrightarrow \End(H^1(A))$ induite par le foncteur $H^1,$ est un morphisme injectif de $F$-algèbres.
\end{prop}
\begin{proof}
(i) Considérons les décompositions suivantes dans $\Ve^{\pm}_F$
\[H^*(A)=\bigoplus_{i=0}^{2g} R(\mathfrak{h}^i(A))  \hspace{0.3cm} \textrm{et}\hspace{0.3cm} H^*(A)=\bigoplus_{i=0}^{2g} H^i(A).\]

Le Théorème \ref{DM} décrit comment l'application $n_A$ agit sur chaque $\mathfrak{h}^i(A)$ via le foncteur motif $M$, on en déduit  
$H^*(n_A) = \bigoplus_{i=0}^{2g} n^i\Id_{R(\mathfrak{h}^i(A))}.$
En particulier pour chaque $i$, on a la décomposition 
\[H^i(A) = \bigoplus_{j=0}^{2g} H^i(A)_j,\]
où $H^i(A)_j$ est le sous-espace propre associé à la valeur propre $n^j$ pour l'action de $H^i(n_A)$. On a donc aussi
 \[R(\mathfrak{h}^i(A)) = \bigoplus_{j=0}^{2g} H^j(A)_i.\]
 
 Remarquons maintenant que si $i$ est impair (respectivement pair) les espaces vectoriels gradués $H^i(A)$ et $R(\mathfrak{h}^i(A))$ sont annulés par une puissance symétrique (respectivement extérieure) assez grande: pour $H^i(A)$ c'est évident car il est concentré en un seul degré et pour $R(\mathfrak{h}^i(A))$ cela vient du Théorème \ref{Ku}(ii). Ceci vaut en particulier pour leurs sous-objets, on en déduit que $H^i(A)_j=0$ quand $i$ et $j$ n'ont pas la même parité.
 
 Remarquons aussi que $R(\mathfrak{h}^1(A))$ est un espace vectoriel de dimension $2g$. En effet, par le Théorème \ref{Ku}(ii), $\Sym^{2g}\mathfrak{h}^1(A)= \mathds{1} (-g)$ et $\Sym^{2g+1}\mathfrak{h}^1(A)=0$,
 ce qui implique 
 \[ \Sym^{2g}R(\mathfrak{h}^1(A)) \neq 0 \hspace{0.3cm} \textrm{et} \hspace{0.3cm}  \Sym^{2g+1}R(\mathfrak{h}^1(A)) = 0.\]
 
 Prouvons maintenant que $R(\mathfrak{h}^1(A)) \subseteq H^1(A)$. Pour tout $j$ impair notons $d_j$ la dimension de $H^j(A)_1$ comme espace vectoriel. Choisissons $v_1,\ldots,v_{2g}\in R(\mathfrak{h}^1(A))$ une base de $ R(\mathfrak{h}^1(A))$ de sorte que chaque élément de cette base appartienne à un $H^j(A)_1$. Utilisons la structure d'algèbre de $H^*(A)$ induite par l'inclusion diagonale $A \hookrightarrow A \times_S A$. Le produit de ces $2g$ éléments appartient à $H^K(A)$, où $K={\sum_{j \, \textrm{impair}}jd_j}.$
 
 Supposons par l'absurde que $R(\mathfrak{h}^1(A)) \not\subseteq H^1(A)$. Ceci est équivalent au fait qu'il y ait un $d_j$ non nul pour un indice $j\neq 1$. Dans ce cas ${\sum_{j \, \textrm{impair}}jd_j}> 2g$ donc $H^K(A)=0.$ En particulier le produit des éléménts $v_1,\ldots,v_{2g}$  est nul. Par le Théorème \ref{Ku}(iii) ceci impliquerait que $ \Sym^{2g} R(\mathfrak{h}^1(A))=0$. On a donc une contradiction, on conclut $R(\mathfrak{h}^1(A)) \subseteq H^1(A)$. 
 
 À ce moment-là,  par le Théorème \ref{Ku}(i) et (iii), on a  que $R(\mathfrak{h}^i(A)) \subseteq H^i(A)$. Comme on a 
 $\bigoplus_{i=0}^{2g} R(\mathfrak{h}^i(A))=H^*(A)=\bigoplus_{i=0}^{2g} H^i(A)$, ces inclusions suffisent à conclure les égalités $R(\mathfrak{h}^i(A)) = H^i(A)$.
 
 \

 (ii) Nous avons démontré plus haut dans la preuve que $R(\mathfrak{h}^1(A))$ est un espace vectoriel de dimension $2g$, on déduit donc (ii) de (i).
 
 \
 
 (iii) C'est une conséquence de (i) et (ii), par \cite[Thm 3.1]{Klei}.
 
 \

 (iv) Tout d'abord c'est un morphisme d'algèbre par (i) et par la Proposition \ref{Ki}. Pour l'injectivité il y a au moins deux manières assez différentes de procéder. 
 
 Par le point (iii) c'est une question qui est indépendante de $H^*,$ on peut alors choisir une cohomologie $\ell$-adique.
 
 En alternative, on remarque que, comme la $F$-algèbre $\End(A)^{\op}\otimes_{\Z} F$ est semisimple de dimension finie, le noyau est engendré par un élément $f$ qui vérifie $f^2=f.$ D'autre part par \cite[Prop. 7.2(ii)]{Kim}, $f$ doit être nilpotent, donc $f=0$.

\end{proof}

\section{Le groupe de Lefschetz $\Lef(A)$}
Dans cette section nous fixons une cohomologie de Weil $H^{*}$ à coefficients dans $F$ (Définition \ref{def coho Weil}) et soit $R$ le foncteur de réalisation induit (\ref{def real}(iii)). Soit $\mathfrak{h}^1(A)$ le motif défini dans le Théorème \ref{DM}, il vérifie  $R(\mathfrak{h}^1(A))=H^1(A)$ (Proposition \ref{coho abel}(i)). Nous rappelons ici la construction du groupe de Lefschetz $\Lef(A);$ c'est par définition un sous-groupe algébrique de $\GL(H^1(A))$. Pour plus de détails voir \cite[\S 2]{Milne}.
\begin{lem}\label{inv ind}
Fixons  $\PL: \mathfrak{h}^1(A) \isocan \mathfrak{h}^{1}(A)^{\vee}(-1)$ un isomorphisme de Poincaré-Lefschetz (corollaire \ref{coro Ku}). Alors:

(i) on a l'égalité  \[\Hom (\mathfrak{h}^1(A),\mathfrak{h}^{1}(A)^{\vee}(-1))=[(\mathfrak{h}^1(\End(A)^{\op})\otimes_{\Z}F)^{\vee}(-1)] \circ \PL ,\]

(ii) l'anti-involution de l'algèbre $\End_{H^1(\End(A)^{\op})}(H^1(A))$ définie par \[f \mapsto f^*:= R(\PL^{-1}) \circ[ f^{\vee}(-1)] \circ R(\PL)\] est indépendante du choix de $\PL$.
\end{lem}
\begin{proof}
Le point (i) vient du corollaire \ref{coro Ku} et de la Proposition \ref{Ki}; puis (i) implique (ii) via  la Proposition \ref{coho abel}(i).
\end{proof}
\begin{defin}\label{gpe lef}
Le groupe de Lefschetz $\Lef(A)$ est 
 \[\Lef(A)=\{f \in \End_{H^1(\End(A)^{\op})}(H^1(A)), \hspace{0.2cm} f^*f=\Id\},\]
 où $f^*$ est défini par le Lemme \ref{inv ind}(ii).
\end{defin}

\begin{rem}\label{notaz} (i) Le groupe de Lefschetz est un groupe réductif, a priori non connexe \cite[Summary p.655]{Milne}.

(ii) Dans la littérature (par exemple dans \cite{Milne}) le groupe $\Lef(A)$ ne possède pas de nom particulier et est souvent noté $S(A)$. On appelle plutôt groupe de Lefschetz (souvent noté $L(A)$) un groupe légérement plus grand qui est défini en demandant que  $f^*f$ soit une homothétie. La seule différence entre ces deux groupes est que  $L(A)$ garde une information sur le poids des représentations, ce qui n'a pas d'importance ici.

\end{rem}

\begin{prop}\label{n=1}
Le foncteur $\mathfrak{h}^1$ et la réalisation induisent des isomorphismes de $F$-algèbres
\[\End(A)^{\op}\otimes_{\Z}F \isocan \End_{\CHM(S)_F}(\mathfrak{h}^1(A)) \isocan \End_{\Lef(A)}(H^1(A)).\]
\end{prop}
\begin{proof}
Grâce aux Propositions \ref{Ki} et  \ref{coho abel}(iv), il reste à voir la surjectivité de la deuxième application. Il suffit alors de comparer la dimension des algèbres. Ceci peut se faire après extensions des scalaires à $\overline{F}$, une clôture algébrique de $F$. Dans ce cas l'algèbre à gauche devient un produit d'algèbres de matrices
\[\End(A)^{\op}\otimes_{\Z} \overline{F}\cong \prod_{i\in I} M_{n_i \times n_i }(\overline{F}).\] En particulier son action décompose $H^1(A)\otimes_F \overline{F}$ en  somme de sous-$\Lef(A)\times_F \overline{F}$-représentations $H^1(A)\otimes_F \overline{F}\cong \oplus_{i\in I} V_i^{\oplus n_i}.$

Par le Lemme de Schur il suffit de voir que les $V_i$ sont des $\Lef(A)\times_F \overline{F}$ représentations irréductibles non isomorphes entre elles, or ceci se déduit de la description explicite des $V_i$ due à Milne \cite[\S 2]{Milne} (voir aussi la Remarque \ref{maledettomilne}).
\end{proof}
\begin{cor}\label{coro n=1}
La réalisation induit les isomorphismes de $F$-espaces vectoriels

\[\begin{array}{r@{\isocan}l}
\Hom_{\CHM(S)_F} (\mathfrak{h}^1(A) \otimes \mathfrak{h}^{1}(A), \mathbb{L}) & \Hom_{\Lef(A)} (H^1(A) \otimes H^{1}(A), F(-1))\hspace{0.2cm} \hspace{0.2cm} \textrm{et}\\

\Hom_{\CHM(S)_F} (\mathbb{L}, \mathfrak{h}^1(A) \otimes \mathfrak{h}^{1}(A)) &  \Hom_{\Lef(A)} (F(-1), H^1(A) \otimes H^{1}(A)),
\end{array}\]
en imposant que  $\Lef(A)$ agisse sur $F(-1)$ par l'action triviale.
\end{cor}
\begin{proof}
Par la proposition précédente on a:
\[\begin{array}{r@{\isocan}l}
\Hom_{\CHM(S)_F} (\mathfrak{h}^1(A) \otimes \mathfrak{h}^{1}(A)^{\vee}, \mathds{1}) & \Hom_{\Lef(A)} (H^1(A) \otimes H^{1}(A)^{\vee}, F)\\

\Hom_{\CHM(S)_F} (\mathds{1}, \mathfrak{h}^1(A) \otimes \mathfrak{h}^{1}(A)^{\vee}) &  \Hom_{\Lef(A)} (F, H^1(A) \otimes H^{1}(A)^{\vee}),
\end{array}\]
d'où
\[\begin{array}{r@{\isocan}l}
\Hom_{\CHM(S)_F} (\mathfrak{h}^1(A) \otimes \mathfrak{h}^{1}(A)^{\vee}(-1), \mathbb{L}) & \Hom_{\Lef(A)} (H^1(A) \otimes H^{1}(A)^{\vee}(-1), F(-1))\\

\Hom_{\CHM(S)_F} (\mathbb{L}, \mathfrak{h}^1(A) \otimes \mathfrak{h}^{1}(A)^{\vee}(-1)) &  \Hom_{\Lef(A)} (F(-1), H^1(A) \otimes H^{1}(A)^{\vee}(-1)).
\end{array}\]
Pour conclure il suffit de remarquer que, par construction, la réalisation d'un isomorphisme de \og Poincaré-Lefschetz\fg{} (corollaire \ref{coro Ku}), $R(\PL): H^1(A) \isocan H^{1}(A)^{\vee}(-1)$,   est $\Lef(A)$-équivariant.
\end{proof}
\begin{rem}\label{maledettomilne}
\begin{enumerate}
\item Les résultats de \cite{Milne} sont rédigés dans le cas des variétés abéliennes définies sur un corps algébriquement clos, mais en fait la section \S2 de loc. cit. n'utilise que le fait qu'on dispose de l'action d'une algèbre semisimple avec involution positive sur un $F$-espace vectoriel. Ceci s'applique bien à notre cadre par les Propositions \ref{Ki} et \ref{coho abel}(i).
\item Les algèbres semisimples avec involution positive sont classifiées (classification d'Albert) en type I, II, III et IV (voir par exemple \cite[\S 21]{Mum}). Les résultats qu'on utilise dans \cite[\S 2]{Milne} sont plus exactement: page 651 pour le type I, page 652 pour le type II et page 653 pour les types III et IV.
\end{enumerate}
\end{rem}

\section{L'algèbre $\mathcal{B}_{n,F}$}
Dans cette section $F$ désigne un corps de caractéristique nulle quelconque et \linebreak ${\mathfrak{h}^1(A) \in \CHM(S)_F}$ est le motif défini par le Théorème \ref{DM}.
\begin{lem}\label{gen ind}
Fixons  $\PL: \mathfrak{h}^1(A) \isocan \mathfrak{h}^{1}(A)^{\vee}(-1)$ un isomorphisme (corollaire \ref{coro Ku}). Soient ${\pi:  \mathfrak{h}^1(A) \otimes \mathfrak{h}^{1}(A) \rightarrow \mathbb{L}}$ le morphisme qui correspond à $\PL$ par adjonction  et  ${\iota: \mathbb{L}   \rightarrow \mathfrak{h}^1(A) \otimes \mathfrak{h}^{1}(A)}$ celui qui correspond à $\PL^{-1}$. Alors:

(i) on a les égalités  \[\begin{array}{r@{=}l}
\Hom (\mathfrak{h}^1(A) \otimes \mathfrak{h}^{1}(A), \mathbb{L}) &  \pi \circ [\mathfrak{h}^1(\End(A)^{\op}\otimes_{\Z}F)\otimes \Id_{\mathfrak{h}^{1}(A)}] \hspace{0.2cm} \textrm{et}\\

\Hom (\mathbb{L} ,\mathfrak{h}^1(A) \otimes \mathfrak{h}^{1}(A)) & [\mathfrak{h}^1(\End(A)^{\op}\otimes_{\Z}F)\otimes \Id_{\mathfrak{h}^{1}(A)}] \circ \iota ,
\end{array}
\]

(ii) on a l'égalité $\pi \circ \iota= 2g \cdot \Id_{\mathbb{L}},$

(iii) le morphisme $P=  \frac{1}{2g}  \iota \circ \pi \in \End(\mathfrak{h}^1(A) \otimes \mathfrak{h}^{1}(A))$ est un projecteur définissant $\mathbb{L}$ à l'intérieur de $\mathfrak{h}^1(A) \otimes \mathfrak{h}^{1}(A),$ c'est-à-dire $P^2=P$ et $\Im P \cong \mathbb{L}$.
\end{lem}
\begin{proof}
(i) On le déduit du Lemme \ref{inv ind}(i) par adjonction.

(ii) Comme $\End_{\CHM(S)_F}(\mathbb{L})=F\cdot \Id_{\mathbb{L}},$ l'énoncé peut se vérifier après réalisation. Il découle alors de la Proposition \ref{coho abel}(ii) et du fait suivant. Soient $V$ un $F$-espace vectoriel de dimension $d$ et $f: V \isocan V^{\vee}$ un isomorphisme. Soient $\pi_f: V\otimes V \longrightarrow F$ le morphisme qui correspond à $f$ par adjonction  et $\iota_f: F   \rightarrow V \otimes V$ celui qui correspond à $f^{-1}$. Alors $\pi_f \circ \iota_f=d \cdot \Id.$

(iii) C'est une conséquence formelle de (ii).
\end{proof}
\begin{defin}\label{def Bn}
Pour tout naturel $n\geq 1$ notons $ \mathcal{B}_{n,F}\subseteq\End_{\CHM(S)_F}(\mathfrak{h}^1(A)^{\otimes n})$  la sous-$F$-algèbre engendrée par 
\begin{itemize}
\item[(a)] le groupe des permutations $\mathcal{S}_n$,
\item[(b)] l'anneau $\mathfrak{h}^1(\End(A)^{\op})\otimes \Id_{\mathfrak{h}^1(A)}^{\otimes (n-1)}$,
\item[(c)] le morphisme $P\otimes \Id_{\mathfrak{h}^1(A)}^{\otimes (n-2)},$ quand $n\geq 2$.
\end{itemize}
Ici $P =\frac{1}{2g} \iota \circ \pi$ est défini comme dans le Lemme \ref{gen ind}(iii).
\end{defin}

\begin{rem} Comme on le verra dans le Théorème \ref{thm cle}, l'algèbre \og motivique\fg{} $\mathcal{B}_{n,F}$ fournit un relévement explicite de l'algèbre \og cohomologique\fg{} $\End_{\Lef(A)}(H^1(A,F)^{\otimes n})$. Il ne nous semble pas clair a priori que $\mathcal{B}_{n,F}$ soit la bonne algèbre à considérer pour obtenir un tel relévement. Sa définition a été trouvée en étudiant des exemples (notamment le schéma abélien universel au-dessus d'une surface modulaire de Picard).
\end{rem}

\begin{rem}\label{rem Bn} (i) Grâce à la présence du groupe des permutations (a), la définition précédente est équivalente si on remplace (b) par l'action de $\End(A)^{\op}$ dans une autre coordonnée, ou encore par l'action de l'anneau $\mathfrak{h}^1(\End(A)^{\op}) \times \cdots \times \mathfrak{h}^1(\End(A)^{\op})$ \linebreak($n$ fois).

(ii) Grâce à la présence de l'anneau des endomorphismes (b) et par le Lemme \ref{gen ind}(i), la définition précédente est équivalente si dans (c) on remplace $P$ par l'ensemble $\Hom (\mathbb{L} ,\mathfrak{h}^1(A) \otimes \mathfrak{h}^{1}(A)) \circ \Hom (\mathfrak{h}^1(A) \otimes \mathfrak{h}^{1}(A), \mathbb{L})$.
En particulier la définition ne dépend pas des choix faits pour définir $P$.

(iii) Soit $F'$ une extension du corps $F$. Sous l'identification canonique de $\CHM(S)_{F'}$ avec l'enveloppe pseudo-abélienne de $\CHM(S)_{F}\otimes_F F'$ on a l'égalité \[\mathcal{B}_{n,F'}=\mathcal{B}_{n,F}\otimes_F F'.\] En particulier $ \mathcal{B}_{n,F}= \mathcal{B}_{n,\Q}\otimes_{\Q} F.$

\end{rem}

\section{Théorème de décomposition}\label{Théorème de décomposition}
Dans cette section nous fixons une cohomologie de Weil $H^*$ à coefficients dans $F$ (Définition \ref{def coho Weil}). Soient $R$ le foncteur de réalisation (\ref{def real}(iii)) et $\Lef(A)$ le groupe de Lefschetz (Définition \ref{gpe lef}) qu'on en déduit. Fixons un entier $n\geq 1$ quelconque et soit $\mathcal{B}_{n,F}$ l'algèbre définie dans \ref{def Bn}.
\begin{thm}\label{thm cle}
La réalisation induit un isomorphisme d'algèbres  
\[R:\mathcal{B}_{n,F} \isocan \End_{\Lef(A)}(H^1(A)^{\otimes n}).\]
Ceci fournit un relèvement explicite de l'algèbre $\End_{\Lef(A)}(H^1(A)^{\otimes n})$ dans $\CHM(S)_F$. En particulier, toute décomposition de $H^1(A)^{\otimes n}$ en tant que $\Lef(A)$-représentation se relève canoniquement en une décomposition de $\mathfrak{h}^1(A)^{\otimes n}$ dans $\CHM(S)_F$ et de plus, deux sous-représentations de $H^1(A)^{\otimes n}$ isomorphes se relèvent en deux motifs isomorphes.
\end{thm}
\begin{proof}
L'injectivité est prouvée dans la section \ref{scindage}. Pour la surjectivité, étendons les scalaires de $F$ à $\overline{F}$ une clôture algébrique de $F$. Dans ce cas le groupe $\Lef(A)\times_F \overline{F}$ devient isomorphe à un produit de groupes classiques 
\[\Lef(A)\times_F \overline{F}\cong \prod_{j\in J_l} G^l_j  \times \prod_{j\in J_s} G^s_j \times \prod_{j\in J_o} G^o_j\]
où $G^l_j$ est un groupe général linéaire pour chaque $j\in J_l$, $G^s_j$ est un groupe symplectique pour chaque $j\in J_s$, et $G^o_j$ est un groupe orthogonal pour chaque $j\in J_o$ \cite[Summary p.655]{Milne}. Soient $W^l_j$ la représentation standard  de $G^l_j$, $W^s_j$ la représentation standard  de $G^s_j$ et $W^o_j$ la représentation standard  de $G^o_j$. Notons $\mathcal{E}$ l'ensemble 
\[\mathcal{E}=\{W^l_j\}_{j \in J_l} \cup   \{(W^l_j)^{\vee}\}_{j \in J_l} \cup  \{W^s_j\}_{j \in J_s} \cup  \{W^o_j\}_{j \in J_o} .\] 
Dans la suite un élément de l'ensemble $\mathcal{E}$ sera vu comme représentation de $\Lef(A)\times_F \overline{F}$ via la décomposition $\Lef(A)\times_F \overline{F}\cong \prod_{j\in J_l} G^l_j  \times \prod_{j\in J_s} G^s_j \times \prod_{j\in J_o} G^o_j$, en décrétant que tous les groupes de la décomposition sauf un agissent trivialement. Milne démontre que l'ensemble des sous-$\Lef(A)\times_F \overline{F}$-représentations irréductibles de $H^1(A)\otimes_F \overline{F}$ (à isomorphisme près)  coïncide avec l'ensemble $\mathcal{E}$ \cite[\S 2]{Milne}, voir aussi la Remarque \ref{maledettomilne}. 

\

Par la Proposition \ref{n=1} et la Remarque \ref{rem Bn}(i), il nous suffit donc de voir que  les éléments des $\overline{F}$-espaces vectoriels de la forme
\[\Hom_{\Lef(A)\times_F \overline{F}}\left(\bigotimes_{1 \leq i\leq n} E^d_i,\bigotimes_{1 \leq i\leq n} E^a_i\right)\]
ont un antécédent dans $\mathcal{B}_{n,F}\otimes_F \overline{F}$ (pour tout choix possible des $E^d_i,E^a_i$ dans $\mathcal{E}$). 

Fixons un tel espace vectoriel $\mathcal{H}=\Hom_{\Lef(A)\times_F \overline{F}}\left(\bigotimes_{1 \leq i\leq n} E^d_i,\bigotimes_{1 \leq i\leq n} E^a_i\right)$. Il suffit de comprendre l'espace vectoriel 
\[\mathcal{H}'=\Hom_{\prod_{j\in J_l} G^l_j  \times \prod_{j\in J_s} G^s_j \times \prod_{j\in J_o} G^o_j}\left(\left(\bigotimes_{1 \leq i\leq n} E^d_i\right) \otimes \left( \bigotimes_{1 \leq i\leq n} (E^a_i)^{\vee}\right), \overline{F}\right).\]

Nous avons une liste explicite de générateurs de $\mathcal{H}'$, par \cite[3.1]{Milne} et \cite{Weyl}. Pour la décrire, considérons les partitions de l'ensemble $E^d_1,\ldots,E^d_n,(E^a_1)^{\vee},\ldots,(E^a_n)^{\vee}$ en $n$ sous-ensembles de cardinalité $2$, telles que chaque sous ensemble soit formé soit de deux représentations isomorphes à  $W^s_j$ (pour un même $j$), soit de deux représentations isomorphes à  $W^o_j$ (pour un même $j$), soit d'une représentation isomorphe à $W^l_j$ et une isomorphe à $(W^l_j)^{\vee}$ (toujours pour un même $j$). Pour chacune de tels couples il y a un accouplement standard et considérons le produit tensoriel de ces $n$ accouplements. Il appartient à $\mathcal{H}'$, et mieux, $\mathcal{H}'$ est engendré par les morphismes ainsi construits. Il reste à voir que: si $p' \in \mathcal{H}'$ est un morphisme ainsi construit et $p \in \mathcal{H}$ est celui qui lui correspond par adjonction, alors $p$ a un antécédent dans $\mathcal{B}_{n,F}\otimes_F \overline{F}$.

Quitte à composer avec des permutations (ce que l'on peut faire par \ref{def Bn}(a)) on peut supposer qu'il existe un $0 \leq r\leq n$ tel que la partition soit ainsi faite
\[(E^d_1,E^{a \vee}_1),\ldots ,(E^d_r,E^{a \vee}_r),(E^d_{r+1},E^d_{r+2}),(E_{r+1}^{a \vee},E^{a \vee}_{r+2}), \ldots ,(E^d_{n-1},E^d_{n}),(E^{a \vee}_{n-1},E^{a \vee}_{n}).\]

Dans ce cas $p$ est de la forme suivante 
\[p=\Id_{E^d_1\otimes \ldots \otimes E^d_r}\otimes  (\iota_{r+1,r+2}\circ \pi_{r+1,r+2}) \otimes \ldots \otimes (\iota_{n-1,n}\circ \pi_{n-1,n}), \]
avec $\pi_{i,i+1}: E^d_{n-1} \otimes E^d_{n} \rightarrow \overline{F}$ et $\iota_{i,i+1}: \overline{F} \rightarrow  E^a_{n-1} \otimes E^a_{n}$ (de tels morphismes sont uniques à un scalaire près). On peut donc conclure par la Remarque \ref{rem Bn}(ii).
\end{proof}

\section{Résultats de scindage}\label{scindage}
 Soient  $M:\Var /S \rightarrow \CHM(S)_F$ le foncteur motif (\ref{Motifs de Chow relatifs}) et $M^{\ab}$ sa restriction à la catégorie des schémas abéliens $\Ab_S$ (sous-catégorie non pleine de $\Var /S$).
 
  Soit $\NumM(S)_F$ la catégorie des motifs numériques définie en remplaçant l'équivalence rationnelle par l'équivalence numérique. Soient  $\CHM(S)^{\ab}_F\subset \CHM(S)_F$ et $\NumM(S)^{\ab}_F \subset \NumM(S)_F$ les sous-catégories pleines $F$-linéaires, tensorielles, symétriques, rigides et pseudo-abéliennes engendrées par les motifs des schémas abéliens. Par construction, on dispose canoniquement d'un foncteur plein \linebreak{$F$-linéaire}, tensoriel, symétrique,  et qui commute aux twists de Tate
\[\mathcal{N}:\CHM(S)_F^{\ab} \longrightarrow \NumM(S)^{\ab}_F.\]

Le \textit{Théorème de scindage} qui suit est dû à O'Sullivan\footnote{Ce théorème de scindage est énoncé à page 71, à l'intérieur de la preuve de \cite[thm. 6.1.1]{O'S}. C'est l'ingrédient clé pour la preuve de ce dernier. En effet \cite[thm. 6.1.1]{O'S} n'est rien d'autre que l'énoncé analogue formulé avec les groupes des Chow au lieu des motifs de Chow.

Sans la condition $M^{\ab}= \mathcal{I} \circ  \mathcal{N} \circ M^{\ab}$, la section n'est plus unique et dans ce cas c'est un résultat de André et Kahn \cite{AK}.} \cite{O'S}.

\begin{thm}\label{thm sec}
Soit
$\mathcal{N}:\CHM(S)_F^{\ab} \longrightarrow \NumM(S)^{\ab}_F,$ le foncteur ci-dessus.
Il existe un unique foncteur 
\[\mathcal{I}:\NumM(S)^{\ab}_F \longrightarrow \CHM(S)_F^{\ab}\]
 $F$-linéaire, tensoriel, symétrique, et qui commute aux twists de Tate, vérifiant
 \[\mathcal{N} \circ  \mathcal{I}=\Id_{\NumM(S)^{\ab}_F}\hspace{0.2cm} \textrm{et} \hspace{0.2cm} M^{\ab}= \mathcal{I} \circ  \mathcal{N} \circ M^{\ab}.\]
\end{thm}

\begin{cor}
Soient  $\mathcal{N}$ et $\mathcal{I}$ comme dans le Théorème \ref{thm sec}. Soient  $X \in \CHM(S)_F^{\ab}$ un motif et $\mathcal{B}\subset\End(X)$ une $F$-algèbre engendrée par des endomorphismes qui sont dans l'image $\mathcal{I}$. Alors  $\mathcal{N}_{|\mathcal{B}}: \mathcal{B} \rightarrow \End(\mathcal{N}(X)) $ est  un morphisme injectif de $F$-algèbres, et a fortiori, pour toute réalisation $R$ (\ref{def real}(iii)), \[R_{|\mathcal{B}}: \mathcal{B} \rightarrow \End(R(X)) \] l'est aussi.
\end{cor}
L'injectivité du morphisme de $F$-algèbres $R:\mathcal{B}_{n,F} \isocan \End_{\Lef(A)}(H^1(A)^{\otimes n})$ du Théorème \ref{thm cle} se déduit donc de la partie (vi) du lemme suivant.

\begin{lem}
Soient $\mathcal{I}$ et $\mathcal{N}$ comme dans le Théorème \ref{thm sec} et $\mathfrak{h}^1(A)$ est le motif défini dans le Théorème \ref{DM}, alors:

(i)   on a l'égalité $\mathcal{I} \circ \mathcal{N}(\mathfrak{h}^1(A))=\mathfrak{h}^1(A)$,

(ii) la restriction de $\mathcal{N}$ à l'algèbre $\End(\mathfrak{h}^1(A))$ est injective,

(iii) les restrictions de $\mathcal{N}$ aux espaces vectoriels $\Hom (\mathbb{L}, \mathfrak{h}^1(A) \otimes \mathfrak{h}^{1}(A))$ et \linebreak $ \Hom (\mathfrak{h}^1(A) \otimes \mathfrak{h}^{1}(A), \mathbb{L})$  sont injectives,

(iv) l'algèbre $\End(\mathfrak{h}^1(A))$ est dans l'image de $\mathcal{I}$,

(v) les espaces vectoriels $\Hom (\mathbb{L}, \mathfrak{h}^1(A) \otimes \mathfrak{h}^{1}(A))$ et  $ \Hom (\mathfrak{h}^1(A) \otimes \mathfrak{h}^{1}(A), \mathbb{L})$  sont dans l'image de $\mathcal{I}$,

(vi) les générateurs de l'algèbre $\mathcal{B}_{n,F}$ (Définition \ref{def Bn}) sont dans l'image de $\mathcal{I}$.
\end{lem}
\begin{proof}
(i)   Les $\mathfrak{h}^i(A)$ sont caractérisés par la propriété que, pour tout $n$, l'endomorphisme $M(n_A)$ restreint à $\mathfrak{h}^i(A)$ vaut $n^i \Id$ (Théorème \ref{DM}), or, par le Théorème \ref{thm sec}, on a $M(n_A)= \mathcal{I} \circ  \mathcal{N} \circ M(n_A).$

(ii) Par la Proposition \ref{Ki}, $\End(\mathfrak{h}^1(A))$ est une $F$-algèbre semisimple de dimension finie, le noyau est donc engendré par un élément $f$ qui vérifie $f^2=f.$ D'autre part par \cite[Prop. 7.2(ii)]{Kim}, $f$ doit être nilpotent, donc $f=0$. 
 
(iii) Ceci découle de (ii) en utilisant le corollaire \ref{coro Ku}. 

(iv) Par (i) on a un endomorphisme de $F$-algèbres \[\mathcal{I} \circ \mathcal{N}_{|\End(\mathfrak{h}^1(A))}: \End(\mathfrak{h}^1(A))\longrightarrow \End(\mathfrak{h}^1(A)).\] Par (ii) c'est un morphisme injectif. Par la Proposition \ref{Ki} la $F$-algèbre $\End(\mathfrak{h}^1(A))$ est de dimension finie, donc le morphisme est aussi surjectif (ou en alternative, comme le foncteur $\mathcal{N}$ est plein par construction, $\mathcal{I} \circ \mathcal{N}_{|\End(\mathfrak{h}^1(A))}$ est injectif si et seulement s'il est surjectif).

(v) Analogue à (iv) en remplaçant (ii) par (iii).

(vi) Considérons la liste des générateurs de $\mathcal{B}_{n,F}$ (Définition \ref{def Bn}): (a) le groupe symétrique, (b) l'action provenant de $\End(A)^{\op}$, (c) le projecteur $P=\frac{1}{2g}\iota \circ \pi.$ Alors (a) est automatique parce que le foncteur $\mathcal{I}$ est symétrique, (b) se déduit de (iv) et (c) de (v).

\end{proof}
\begin{rem} Dans \cite{O'S} les cycles appartenant à l'image du foncteur $\mathcal{I}$ sont appelés \og symmetrically distinguished\fg. O'Sullivan montre qu'il est possible de donner une définition alternative de ces cycles qui ne fasse pas intervenir le foncteur $\mathcal{I}$. Cette définition \cite[page 2]{O'S} demande que certains sous-espaces de groupes de Chow construits à partir d'un cycle donné n'intersectent pas le noyau du foncteur $\mathcal{N}$. Cette condition nous semble, dans la pratique, moins facile à être vérifiée.
\end{rem}
\begin{rem}\label{alternative}
Dans la thèse de l'auteur, l'injectivité du morphisme de $F$-algèbres $R:\mathcal{B}_{n,F} \isocan \End_{\Lef(A)}(H^1(A)^{\otimes n})$ du Théorème \ref{thm cle} est démontrée par des méthodes plus élémentaires, basées sur la théorie classique des invariants, dans le même esprit de la preuve de la surjectivité dans la section  6 (ces méthodes n'utilisent pas le Théorème de scindage \ref{thm sec} de O'Sullivan). Le point-clé est que les relations provenant de la théorie des invariants \cite[Thm 4.2, Thm 5.7, Thm 6.7]{DeCon} ne dépendent essentiellement que de la dimension des espaces vectoriels sous-jacents et  elles sont donc motiviques par le Théorème \ref{Ku}(\ref{dimKu}). Le lecteur intéressé par cette approche pourra consulter \cite{Anc}.

Nous avons pris connaissance de \cite{O'S} après avoir soumis \cite{Anc}. Nous avons choisi ici d'utiliser \cite{O'S} pour plusieurs raisons. Premièrement, nous voulions mettre en lumière ce travail remarquable et expliquer le lien avec le nôtre. Deuxièmement, il nous semble utile de donner un exemple explicite de cycles dans l'image $\mathcal{I}$ (à savoir l'algèbre $\mathcal{B}_{n,F}$). À notre connaissance ceci est le premier exemple présent dans la littérature.  
\end{rem}
\begin{rem}
Le résultat suivant est dû à Moonen \cite[Corollary 9.4]{Moo} (et répond à une question de Voisin): si un cycle d'une variété abélienne peut être écrit comme combinaison linéaire  d'intersections de diviseurs symétriques, alors il est numériquement trivial si et seulement s'il est rationnellement trivial. En suivant les méthodes présentées dans cette section, on voit qu'on peut déduire ce résultat du Théorème \ref{thm sec} de O'Sullivan. Il est  aussi possible de déduire ce résultat directement de l'énoncé du Théorème \ref{thm cle} (et donc, en suivant la Remarque \ref{alternative}, de la théorie classique des invariants).
\end{rem}

\section{Conséquences}
Dans cette section nous donnons des conséquences du Théorème \ref{thm cle}. Le Théorème \ref{cor princ} est une généralisation du Théorème \ref{thm intro} de l'introduction. La deuxième partie de la section est dévouée aux variétés de Shimura de type PEL. Nous démontrons le Théorème \ref{automorphe} de l'introduction ainsi qu'un résultat relié (Théorème \ref{thm rep}).
\begin{npar}{L'algèbre $B_{i,r}$}\label{Bir}
Nous gardons ici les notations rappelés au début de la section 6. Pour tous naturels $i$ et $r$ considérons les identifications qui suivent:
\[\mathfrak{h}^i(A^r)= \Sym^i \mathfrak{h}^1(A^r)\]
du Théorème \ref{Ku}(i),
\[\mathfrak{h}^1(A)^{\oplus r}= \mathfrak{h}^1(A^r)\]
de la Proposition \ref{Ki}, et 
\[(\mathfrak{h}^1(A)^{\oplus r})^{\otimes i}=(\mathfrak{h}^1(A)^{\otimes i})^{\oplus r^i}.\]
Sous ces identifications, l'algèbre $\mathcal{B}_{i,F}^{\oplus r^i}\subseteq\End_{\CHM(S)_F}((\mathfrak{h}^1(A)^{\otimes i})^{\oplus r^i})$ est canoniquement une sous-algèbre de $\End_{\CHM(S)_F}((\mathfrak{h}^1(A^r))^{\otimes i})$. De plus, le projecteur 
  \[P_i=\sum_{\sigma \in \mathcal{S}_i} \sigma\]
  définissant $\Sym^i \mathfrak{h}^1(A^r)$ comme facteur direct de $(\mathfrak{h}^1(A^r))^{\otimes i}$ appartient à l'algèbre. Définissons
  \[B_{i,r}\otimes_{\Q} F=P_i \circ \mathcal{B}_{i,F}^{\oplus r^i} \circ P_i, \]
  sous-algèbre de  $\mathcal{B}_{i,F}^{\oplus r^i}$. La notation est justifiée par la Remarque \ref{rem Bn}(iii). Sous les identifications ci-dessus on a l'inclusion canonique
  \[B_{i,r}\otimes_{\Q} F \subset \End_{\CHM(S)_F}(\mathfrak{h}^i(A^r)).\]
\end{npar}
\begin{thm}\label{cor princ}
La réalisation induit un isomorphisme d'algèbres  
\[R:B_{i,r}\otimes_{\Q} F \isocan \End_{\Lef(A)}(H^i(A^r)),\]
pour tous $i$ et $r$.
Ceci fournit un relèvement explicite de l'algèbre $\End_{\Lef(A)}(H^i(A^r))$, en particulier toute décomposition de $H^i(A^r)$ en tant que {$\Lef(A)$-représentation} se relève canoniquement  en une décomposition de $\mathfrak{h}^i(A^r)$ dans $\CHM(S)_F$ et de plus, deux sous-représentations de $H^i(A^r)$ isomorphes se relèvent en deux motifs isomorphes.
\end{thm}
 \begin{npar}{Données PEL}
 Dorénavant nous traitons le cas des variétés de Shimura de type PEL et nous suivons les notations du Théorème \ref{automorphe}. Pour les généralités sur les données PEL le lecteur pourra consulter \cite{Kott}. Rappelons qu'une donnée PEL est en particulier un $\Q$-espace vectoriel $V$ muni de l'action d'une algèbre semisimple $B$ (avec une anti-involution positive) et d'une forme bilinéaire antisymétrique. Le groupe $G$ de la donnée est défini comme le groupe  des endomorphismes de $V$ commutant à $B$ et respectant la forme bilinéaire à un scalaire près. On a en particulier une suite exacte
 \[1 \rightarrow G_1 \rightarrow G \rightarrow \mathbb{G}_{m,\Q}\rightarrow 1,\] où $G_1$ est le sous-groupe de $G$ des endomorphismes qui respectent la forme bilinéaire (comparer avec la Définition \ref{gpe lef} de groupe de Lefschetz). On notera $\Q(n)$ l'espace  vectoriel de dimension $1$ sur lequel $G$ agit via l'action de poids $n$ de $\mathbb{G}_{m,\Q}.$ L'opération de tensoriser par $\Q(n)$ dans $\Rep(G)$ est appelée \textit{twists de Tate}.
 \end{npar}
 
 \begin{npar}{Construction canonique \cite[\S 1.18]{Pi}}
 Soit $F$ un corps de caractéristique zéro. Le foncteur de construction canonique
 \[\mu: \Rep (G_F) \longrightarrow \VHS(S_0(\C))_F\]
est un  foncteur $F$-linéaire, tensoriel et symétrique qui commute aux twists de Tate et qui envoie  $V^{\vee}$ sur $R^1 f_*\Q_{A_0}$. 

Rappelons aussi qu'on dispose d'un foncteur de \og réalisation de Hodge\fg
 \[\CHM(S_0)_F \rightarrow \VHS(S_0(\C))_F,\] 
 voir l'Exemple \ref{exe weil}(i).
 \end{npar}
 \begin{prop}
 Suivons les notation du Théorème \ref{automorphe} et soit $F$ un corps de caractéristique zéro, alors:
 
 (i) L'algèbre $B_{i,r}\otimes_{\Q} F$ (\S \ref{Bir}) relève l'algèbre $\mu(\End_{G_F}(\Lambda^i((V^{\vee})^{\oplus r})))$.
 
 (ii) L'algèbre $B_{n,F}$ (Définition \ref{def Bn}) relève l'algèbre $\mu(\End_{G_F}((V^{\vee})^{\otimes n}))$.
 \end{prop}
 \begin{proof}
Le point (i) est un cas particulier du Théorème \ref{cor princ} et le point (ii) est un cas particulier du Théorème \ref{thm cle}. En effet  le foncteur $\mu$ identifie $G_1$ avec le groupe de Lefschetz $\Lef(A_0)$. Le passage de $G_1$ à $G$ est anodin, car $G$ est engendré par $G_1$ et les homothéties $\mathbb{G}_m \subset G$.
\end{proof}
\begin{proof}[Démonstration du Théorème \ref{automorphe}]
 L'énoncé se déduit formellement du point (i) de la proposition précedente (voir par exemple les lignes qui suivent le Théorème \ref{thm intro}).
 \end{proof}
 
 \begin{thm}\label{thm rep}
Le foncteur $\mu: \Rep (G_F)\rightarrow \VHS(S_0(\C))_F$  se relève en un  foncteur $F$-linéaire et tensoriel\footnote{Le foncteur n'est pas symétrique à cause de la règle des signes de Koszul.} $\tilde{\mu}: \Rep (G_F)\rightarrow \CHM(S_0)_F$  qui commute aux twists de Tate et qui envoie  $V^{\vee}$ sur $h^1(A)$ et $\End_{\Rep(G_F)}((V^{\vee})^{\otimes n})$ sur $\mathcal{B}_{n,F}$. 
\end{thm}
\begin{rem}
Explicitement, un objet quelconque de la catégorie $\Rep (G_F)$ est isomorphe à un facteur direct d'une somme finie de la forme $\bigoplus_i V^{\otimes a_i} \otimes (V^{\vee})^{\otimes b_i}.$  Ceci découle du fait que $G_{\overline{F}}$ est un produit de groupe classique et $V$ contient les représentations standards de ces groupes \cite[\S 2]{Milne}.
\end{rem}
\begin{proof}
Par la remarque précédente (et par linéarité) il suffit de déterminer l'image de $\Hom_{G_F}( V^{\otimes a} \otimes (V^{\vee})^{\otimes b} ,  V^{\otimes c} \otimes (V^{\vee})^{\otimes d})$. Par adjonction on se réduit à l'étude de $\Hom_{G_F}((V^{\vee})^{\otimes b+c} ,   (V^{\vee})^{\otimes a+d})$. 

Cet espace vectoriel est nul si $b+c \neq a+d$ (regarder l'action des homothétiés $\mathbb{G}_m \subset G$). Si $b+c = a+d$ on conclut en utilisant le point (ii) de la proposition précédente.

\end{proof}

\section{Coda: Motifs associés aux formes modulaires}\label{motif form} 
La motivation originale de ce travail vient du programme de Wildeshaus \cite{Wild3,Wild4} pour la construction de motifs associés aux formes modulaires. Nous souhaitons donner ici un survol de ce programme et expliquer le rôle du Théorème \ref{automorphe}.
 
 Soient $S_0$ le modèle canonique d'une variété de Shimura de type PEL, $A_0$ le schéma abélien universel au-dessus de $S_0$ et $A_0^r$ le $r$-ième produit fibré de $A_0$ au-dessus de $S_0$. Notons $k$ le corps reflexe; les variétés $S_0,A_0$ et $A_0^r$ sont définies sur $k$. On notera  $\DM(k)$ la catégorie triangulée des motifs mixtes sur $k$ définie par Voevodsky et al. \cite{Voe}, $\CHM(S_0)$ et $\CHM(k)$ noteront la catégorie des motifs de Chow au-dessus de $S_0$ et de $k$ respectivement et $M_S(A_0^r)\in \CHM(S_0)$ et $M(A_0^r)\in \DM(k)$ seront le motif de $A_0^r$ au-dessus de $S_0$ et de $k$ respectivement.
 
 \

 Dans \cite{SchMod}, Scholl  étudie le cas des courbes modulaires. Il construit un motif de Chow au-dessus de $k$ sur lequel l'algèbre de Hecke agit, et grâce à cette action il peut découper des sous-motifs  dont les réalisations sont les représentations  $\ell$-adiques (construites par Deligne \cite{Delmod}) associées aux formes modulaires classiques. Ce motif est un facteur direct du motif de $\overline{A_0^r}$, une compactification (lisse) de $A_0^r$ sur laquelle l'algèbre de Hecke agit.
 
Cette approche n'est pas généralisable: quand $S_0$ n'est pas une courbe modulaire on ne connaît pas de compactifications lisses de $A_0^r$ munies de l'action de l'algèbre de Hecke. L'idée de Wildeshaus est de construire directement des motifs purs (de Chow) sur lesquels l'algèbre de Hecke agit.

\

D'après  \cite[Thm 4.3]{Wild3}, sous une condition de \og évitement de poids\footnote{Poids ici est au sens de Bondarko \cite{Bon}. C'est une notion motivique qui relève la notion classique de filtration de poids en Théorie de Hodge ou en cohomologie $\ell$-adique.}\fg, il est possible d'associer à un motif (mixte) dans $\DM(k)$ un motif de Chow canonique qui représente sa partie pure. Dans le cadre géométrique qui nous intéresse, ce motif serait muni d'une action canonique de l'algèbre de Hecke. La condition d'évitement de poids se teste sur le motif bord $\partial M(A_0^r)\in \DM(k)$  de $A_0^r$. Le motif bord d'une variété lisse $X$ est défini dans \cite{WildBound} et mesure le défaut de compacité de la variété. D'après loc. cit., on dispose d'un triangle canonique dans $\DM(k)$, 
\[\partial M(X) \longrightarrow M(X)\longrightarrow M_c(X)\stackrel{+1}{\longrightarrow},\]
où $M_c(X)$ est le motif à support compact de $X$ \cite{Voe}.

Cette condition d'évitement de poids n'est pas vérifiée sur le motif $\partial M(A_0^r)$ tout entier (elle n'est même pas vérifiée après réalisation). D'autre part, d'après  \cite[Theorem 2.2]{Wild5}, tout facteur direct $M_S(A_0^r)^p$ du motif de Chow de  $A_0^r$ (dans $\CHM(S_0)$) induit un triangle canonique
\[\partial M(A_0^r)^p \longrightarrow M(A_0^r)^p\longrightarrow M_c(A_0^r)^p\stackrel{+1}{\longrightarrow}\]
dans $\DM(k)$ ($M(A_0^r)^p$ et $M_c(A_0^r)^p$ ne sont rien d'autre que l'image directe $\alpha_*$ et l'image directe à support compact $\alpha_!$ de $M_S(A_0^r)^p$ le long du morphisme structurel $\alpha: S_0 \longrightarrow k$).

Ici l'intérêt du Théorème \ref{automorphe}: nous disposons des $\partial M(A_0^r)^p$, facteurs directs  de $\partial M(A_0^r)$, pour lesquels on peut tester la condition d'évitement de poids. De plus, leur réalisation est en principe calculable en utilisant \cite{BW}. Des calculs explicites ont été fait dans \cite{Wild4} pour les variétés de Hilbert-Blumenthal et dans \cite{Anc14} pour les surfaces modulaires de Picard. Dans les deux cas la \og majorité \fg{} des $\partial M(A_0^r)^p$ vérifient l'évitement de poids \textbf{après réalisation}, précisement il s'agit des facteurs qui correspondent via le Théorème \ref{automorphe} aux représentations régulières du groupe sous-jacent à la donnée de Shimura (c'est-à-dire les représentations qui ne sont pas sur un mur de la chambre de Weyl). 

Il est raisonnable d'espérer que ces mêmes facteurs vérifient la condition d'évitement de poids  au niveau motivique. Ceci a été démontré dans \cite{Wild4} pour les variétés de Hilbert-Blumenthal. D'après la discussion ci-dessus, la condition d'évitement de poids pour $\partial M(A_0^r)^p$ implique l'existence d'un motif de Chow canonique $\overline{M(A_0^r)^p}\in \CHM(k)$, muni d'une action canonique de l'algèbre de Hecke, et qui se réalise dans la partie pure de la réalisation de $M(A_0^r)^p$.

\bibliographystyle{alpha}

\begin{thebibliography}{Mum08}

\bibitem[AK02]{AK}
Y.~Andr{\'e} and B.~Kahn.
\newblock Nilpotence, radicaux et structures mono\"\i dales.
\newblock {\em Rend. Sem. Mat. Univ. Padova}, 108:107--291, 2002.
\newblock With an appendix by P. O'Sullivan.

\bibitem[Anc12]{Anc}
G.~Ancona.
\newblock Décomposition du motif d'un schéma abélien universel.
\newblock {\em Thèse de doctorat Université Paris XIII}, pages 1--60, 2012.
\newblock Disponible sur
  http://www.math.univ-paris13.fr/$\sim$ancona/doc/these.pdf.

\bibitem[Anc14]{Anc14}
G.~Ancona.
\newblock Degeneration of hodge structures over {P}icard modular surfaces.
\newblock {\em Preprint. Arxiv: http://arxiv.org/pdf/1403.5187.pdf}, pages
  1--18, 2014.

\bibitem[And04]{Andmot}
Y.~Andr{\'e}.
\newblock {\em Une introduction aux motifs (motifs purs, motifs mixtes,
  p\'eriodes)}, volume~17 of {\em Panoramas et Synth\`eses [Panoramas and
  Syntheses]}.
\newblock Soci\'et\'e Math\'ematique de France, Paris, 2004.

\bibitem[Bon10]{Bon}
M.~V. Bondarko.
\newblock Weight structures vs. {$t$}-structures; weight filtrations, spectral
  sequences, and complexes (for motives and in general).
\newblock {\em J. K-Theory}, 6(3):387--504, 2010.

\bibitem[BW04]{BW}
J.~I. Burgos and J.~Wildeshaus.
\newblock Hodge modules on {S}himura varieties and their higher direct images
  in the {B}aily-{B}orel compactification.
\newblock {\em Ann. Sci. \'Ecole Norm. Sup. (4)}, 37(3):363--413, 2004.

\bibitem[CD12]{CDWeil}
D.~C. Cisinski and F.~D{\'e}glise.
\newblock Mixed {W}eil cohomologies.
\newblock {\em Adv. Math.}, 230(1):55--130, 2012.

\bibitem[CH00]{CoHa}
A.~Corti and M.~Hanamura.
\newblock Motivic decomposition and intersection {C}how groups. {I}.
\newblock {\em Duke Math. J.}, 103(3):459--522, 2000.

\bibitem[dCP76]{DeCon}
C.~de~Concini and C.~Procesi.
\newblock A characteristic free approach to invariant theory.
\newblock {\em Advances in Math.}, 21(3):330--354, 1976.

\bibitem[Del71]{Delmod}
P.~Deligne.
\newblock Formes modulaires et repr\'esentations {$l$}-adiques.
\newblock In {\em S\'eminaire {B}ourbaki. {V}ol. 1968/69: {E}xpos\'es
  347--363}, volume 175 of {\em Lecture Notes in Math.}, pages Exp.\ No.\ 355,
  139--172. Springer, Berlin, 1971.

\bibitem[DM91]{DM}
C.~Deninger and J.~Murre.
\newblock Motivic decomposition of abelian schemes and the {F}ourier transform.
\newblock {\em J. Reine Angew. Math.}, 422:201--219, 1991.

\bibitem[Haz85]{Haz}
F.~Hazama.
\newblock Algebraic cycles on certain abelian varieties and powers of special
  surfaces.
\newblock {\em J. Fac. Sci. Univ. Tokyo Sect. IA Math.}, 31(3):487--520, 1985.

\bibitem[Ich91]{Ichi}
T.~Ichikawa.
\newblock Algebraic groups associated with abelian varieties.
\newblock {\em Math. Ann.}, 289(1):133--142, 1991.

\bibitem[Kim05]{Kim}
S.-I. Kimura.
\newblock Chow groups are finite dimensional, in some sense.
\newblock {\em Math. Ann.}, 331(1):173--201, 2005.

\bibitem[Kin98]{Ki}
G.~Kings.
\newblock Higher regulators, {H}ilbert modular surfaces, and special values of
  {$L$}-functions.
\newblock {\em Duke Math. J.}, 92(1):61--127, 1998.

\bibitem[Kle94]{Klei}
S.~L. Kleiman.
\newblock The standard conjectures.
\newblock In {\em Motives ({S}eattle, {WA}, 1991)}, volume~55 of {\em Proc.
  Sympos. Pure Math.}, pages 3--20. Amer. Math. Soc., Providence, RI, 1994.

\bibitem[Kot92]{Kott}
Robert~E. Kottwitz.
\newblock Points on some {S}himura varieties over finite fields.
\newblock {\em J. Amer. Math. Soc.}, 5(2):373--444, 1992.

\bibitem[K{\"u}n93]{Ku2}
K.~K{\"u}nnemann.
\newblock A {L}efschetz decomposition for {C}how motives of abelian schemes.
\newblock {\em Invent. Math.}, 113(1):85--102, 1993.

\bibitem[K{\"u}n94]{Ku1}
K.~K{\"u}nnemann.
\newblock On the {C}how motive of an abelian scheme.
\newblock In {\em Motives ({S}eattle, {WA}, 1991)}, volume~55 of {\em Proc.
  Sympos. Pure Math.}, pages 189--205. Amer. Math. Soc., Providence, RI, 1994.

\bibitem[Mil99]{Milne}
J.~S. Milne.
\newblock Lefschetz classes on abelian varieties.
\newblock {\em Duke Math. J.}, 96(3):639--675, 1999.

\bibitem[Moo14]{Moo}
B.~Moonen.
\newblock On the motive of an abelian scheme with non-trivial endomorphisms.
\newblock {\em J. Reine Angew. Math.}, pages 1--35, 2014.

\bibitem[Mum08]{Mum}
D.~Mumford.
\newblock {\em Abelian varieties}, volume~5 of {\em Tata Institute of
  Fundamental Research Studies in Mathematics}.
\newblock Published for the Tata Institute of Fundamental Research, Bombay,
  2008.
\newblock With appendices by C. P. Ramanujam and Yuri Manin, Corrected reprint
  of the second (1974) edition.

\bibitem[Mur84]{Mur}
V.~K. Murty.
\newblock Exceptional {H}odge classes on certain abelian varieties.
\newblock {\em Math. Ann.}, 268(2):197--206, 1984.

\bibitem[O'S11]{O'S}
P.~O'Sullivan.
\newblock Algebraic cycles on an abelian variety.
\newblock {\em J. Reine Angew. Math.}, 654:1--81, 2011.

\bibitem[Pin90]{Pi}
Richard Pink.
\newblock {\em Arithmetical compactification of mixed {S}himura varieties}.
\newblock Bonner Mathematische Schriften [Bonn Mathematical Publications], 209.
  Universit\"at Bonn Mathematisches Institut, Bonn, 1990.
\newblock Dissertation, Rheinische Friedrich-Wilhelms-Universit{\"a}t Bonn,
  Bonn, 1989.

\bibitem[Rib83]{Ribet}
K.~A. Ribet.
\newblock Hodge classes on certain types of abelian varieties.
\newblock {\em Amer. J. Math.}, 105(2):523--538, 1983.

\bibitem[Sch90]{SchMod}
A.~J. Scholl.
\newblock Motives for modular forms.
\newblock {\em Invent. Math.}, 100(2):419--430, 1990.

\bibitem[Tan82]{Tan}
S.~G. Tankeev.
\newblock Cycles on simple abelian varieties of prime dimension.
\newblock {\em Izv. Akad. Nauk SSSR Ser. Mat.}, 46(1):155--170, 192, 1982.

\bibitem[Voe00]{Voe}
V.~Voevodsky.
\newblock Triangulated categories of motives over a field.
\newblock In {\em Cycles, transfers, and motivic homology theories}, volume 143
  of {\em Ann. of Math. Stud.}, pages 188--238. Princeton Univ. Press,
  Princeton, NJ, 2000.

\bibitem[Wey97]{Weyl}
H.~Weyl.
\newblock {\em The classical groups}.
\newblock Princeton Landmarks in Mathematics. Princeton University Press,
  Princeton, NJ, 1997.
\newblock Their invariants and representations, Fifteenth printing, Princeton
  Paperbacks.

\bibitem[Wil06]{WildBound}
J.~Wildeshaus.
\newblock The boundary motive: definition and basic properties.
\newblock {\em Compos. Math.}, 142(3):631--656, 2006.

\bibitem[Wil09]{Wild3}
J.~Wildeshaus.
\newblock Chow motives without projectivity.
\newblock {\em Compos. Math.}, 145(5):1196--1226, 2009.

\bibitem[Wil11]{Wild5}
J.~Wildeshaus.
\newblock Boundary motive, relative motives and extensions of motives.
\newblock {\em Panoramas et Synth\`eses. Autour des motifs. \'Ecole d'\'et\'e
  franco-asiatique de g\'eom\'etrie alg\'ebrique et de th\'eorie des nombres.
  J.-B. Bost and J.-M. Fontaine eds.}, Vol. II:46 pages, 2011.

\bibitem[Wil12]{Wild4}
J.~Wildeshaus.
\newblock On the interior motive of certain {S}himura varieties: the case of
  {H}ilbert-{B}lumenthal varieties.
\newblock {\em International Mathematics Research Notices},
  2012(10):2321--2355, 2012.

\bibitem[Zar94]{Zar}
Y.~G. Zarhin.
\newblock The {T}ate conjecture for nonsimple abelian varieties over finite
  fields.
\newblock In {\em Algebra and number theory ({E}ssen, 1992)}, pages 267--296.
  de Gruyter, Berlin, 1994.

\end{thebibliography}
\selectlanguage{french}

\end{document}